\documentclass{article}

\usepackage{amsmath,amsthm,amssymb,amsfonts}
\usepackage{paralist}
\usepackage{stmaryrd}
 \usepackage{float}
\usepackage{rotating}
\usepackage{epstopdf}
\usepackage{color}

\usepackage{pgfplots}
\usepackage{pgfplotstable}

\usepackage{enumitem}

\usepackage[boxruled]{algorithm2e}

\usepackage{hyperref}

\newtheorem{theorem}{Theorem}[section]

\newtheorem{lemma}[theorem]{Lemma}
\newtheorem{corollary}[theorem]{Corollary}
\newtheorem{assumption}[theorem]{Assumption}

\newtheorem{myalg}{Algorithm}

\newcommand{\sign}{\text{sign}}

\def \R{\mathbb{R}}

\def \N{\mathbb{N}}
\def \eps{\varepsilon}
\def \Uad{{U_{\mathrm{ad}}}}
\def \tin{{\mathrm{in}}}
\def \tex{{\mathrm{ex}}}

\DeclareMathOperator*{\argmin}{arg\,min}
\numberwithin{equation}{section}

\begin{document}

\title{An inexact iterative Bregman method for optimal control problems\footnote{This work was funded by German Research Foundation DFG under project grant Wa 3626/1-1.}}

\author{Frank P\"orner\footnote{Department of Mathematics,  University of W\"urzburg, Emil-Fischer-Str. 40, 97074 W\"urzburg, Germany, E-mail: frank.poerner@mathematik.uni-wuerzburg.de}
}
\date{\today}
\maketitle

\begin{abstract}
In this article we investigate an inexact iterative regularization method based on generalized Bregman distances of an optimal control problem with control constraints.  We show robustness and convergence of the inexact Bregman method under a regularity assumption, which is a combination of a source condition and a regularity assumption on the active sets.  {We also take the discretization error into account.} Numerical results are presented to demonstrate the algorithm.\\
 

\end{abstract}

\section{Introduction}
We consider an optimization problem of the following form:
\begin{equation}\label{eq:main_problem}\tag{$\textbf{P}$}
 \begin{split}
    \text{Minimize} &\quad \frac{1}{2}\|Su - z\|_Y^2  \\
    \text{such that} &\quad u_a \leq u \leq u_b \quad \text{a.e. in } \Omega.
\end{split}
\end{equation}
Here $\Omega \subseteq \R^n$, $n \geq 1$ is a bounded, measurable set, $Y$ a Hilbert space and $z \in Y$ a given function. The operator $S: L^2(\Omega) \to Y$ is supposed to be linear and continuous and inequality constraints are prescribed on the set $\Omega$. Here, we have in mind to choose $S$ as the solution operator of a linear partial differential equation. The special case where $y=Su$ is defined as the solution of
\begin{equation*}
\begin{alignedat}{3}
 {-}\Delta y &= u \;\; &&\text{in } \Omega\\
 y&=0 &&\text{on } \partial \Omega.
\end{alignedat}
\end{equation*}
will be treated in detail in section \ref{sec:numerics}.

A well-known method to solve \eqref{eq:main_problem} is the proximal point method (PPM) introduced by Martinet \cite{Martinet1970} and developed by Rockafellar \cite{Rockafellar1976}. This method is also known as iterated Tikhonov regularization, see \cite{engl1996,hankegroetsch98, obradovic2015}. The PPM is an iterative method, and the next iterate $u_{k+1}$ is given as the solution of
\begin{equation*}
 \begin{split}
    \text{Minimize} &\quad \frac{1}{2}\|Su - z\|_Y^2  + \alpha_{k+1}\|u-u_k\|_{L^2(\Omega)}^2\\
    \text{such that} &\quad u_a \leq u \leq u_b \quad \text{a.e. in } \Omega,
\end{split}
\end{equation*}
{with some given initial starting value $u_0$.} Here $(\alpha_k)_k$ is a sequence of non-negative real numbers. One can hope to obtain convergence without the additional requirement that the regularization parameters $\alpha_k$ tend to zero. Unfortunately this is not the case in general, since there exists a counter-example by G\"{u}ler \cite{gueler1991}. There only weak convergence is obtained. However this method is well understood, see e.g. \cite{tichatschke2000,tichatschke1998,tichatschke1998b,kaplan1994stable,Solodov00} and the references therein.\\
For the PPM method it is interesting to investigate the robustness with respect to numerical errors. Denote by
\[
\mathcal{P} u_k:= \argmin\limits_{{u \in \Uad}} \frac{1}{2}\|Su - z\|_Y^2  + \alpha_{k+1}\|u-u_k\|_{L^2(\Omega)}^2
\]
the exact solution, which in general cannot be computed exactly. {Since $\alpha_{k+1} > 0$ it is clear that this problem has a unique solution.} Due to numerical errors we only obtain an approximate solution $u_{k+1}$ which satisfies $\|u_{k+1} - \mathcal{P} u_k\|_{L^2(\Omega)} \leq \eps_k$. The sequence $(\eps_k)_k$ can be interpreted as the accuracy of the computed solution. One can hope to achieve convergence of the sequence $(u_k)_k$ if $(\eps_k)_k$ is chosen appropriately. The iterates generated by the proximal point method converge weakly to a solution of \eqref{eq:main_problem} if the condition
\begin{equation}\label{eq:ppm_numerror}
\sum\limits_{j=1}^\infty \frac{\eps_j}{\alpha_j} < \infty
\end{equation}
holds, see \cite{tichatschke1998,Rockafellar1976}. If the state $z$ is not attainable, i.e. there exists no feasible control $u \in \Uad$ such that $Su=z$ holds, the optimal solution might be bang-bang. {Here $\Uad$ is the set of all feasible controls 
\[
		\Uad := \{u \in L^2(\Omega): \; u_a \leq u \leq u_b\}.
\]}
This means it is a linear combination of characteristic functions. Hence the solution may not be in $H^1(\Omega)$ and it is unlikely that a source condition holds in this case, see \cite{wachsmuth2011b}.\\
To handle this non-{attainability} we considered in \cite{wachsmuth2016} an iterative method based on generalized Bregman distances. There, the iterate $u_{k+1}$ is given by the solution of
\begin{equation}\label{eq:intro_1}
\text{Minimize} \quad \frac{1}{2}\|Su-z\|_Y^2 + \alpha_{k+1}  D^{\lambda_{k}}(u,u_{k}),
\end{equation}
where $D^\lambda(u,v) := J(u) - J(v) - (u-v,\lambda)$ is called the (generalized) Bregman distance \cite{Bregman67} associated with a regularization function $J$ with subgradient $\lambda \in \partial J(v)$. Here we have additional freedom in choosing the regularization function $J$. This method was first applied to an image restoration problem, where $J$ was chosen to be the total variation, see \cite{burger2007,osher2005}. Our approach was to incorporate the control constraints into the regularization functional, resulting in
\[
J(u) := \frac{1}{2}\|u\|^2 + I_\Uad(u).
\]
Here $I$ is the indicator function from convex analysis. This choice allowed us to prove strong convergence under a suitable regularity assumption, which allows bang-bang structure and non-attainability, see \cite{wachsmuth2016}. In the case of noisy data $\|z-z^\delta\|\leq \delta$ we established an a-priori stopping rule in \cite{poerner2016b}. 

The aim of this paper is to analyse the robustness of the iterative method presented in \cite{wachsmuth2016} with respect to numerical errors. {We replace the operator $S$ in \eqref{eq:intro_1} by a linear and continuous operator $S_h$ with finite-dimensional range $Y_h\subset Y$. This makes the problem \eqref{eq:intro_1} numerically solvable, but introduces an additional discretization error. If $S$ is the solution operator of a linear elliptic partial differential equation and $Y_h$ is spanned by linear finite elements, then this can be interpreted as the variational discretization in the sense of Hinze \cite{Hinze2005}. }

We aim to establish sufficient conditions on the sequence $(\eps_k)_k$ comparable to \eqref{eq:ppm_numerror}, to ensure convergence. \\

This paper is structured as follows. In section \ref{sec:assumptions} we recall our iterative method, our regularity assumption and some convergence results. {The operator $S_h$ is then introduces in section \ref{section:descr_problem}.} Furthermore we present an a-posteriori error estimator for the {discretized} subproblem, which allows to control the accuracy of the iterates. In section \ref{sec:inexact_bregman} we establish our inexact Bregman iteration and show robustness and convergence results under the presence of numerical errors using our regularity assumption. As an example we consider in section \ref{sec:numerics} the optimal control of the heat equation. {We construct the operator $S_h$ and show its properties.}  Furthermore numerical results are presented for a bang-bang example. Finally conclusions are drawn in section \ref{sec:conclusion}.

\paragraph*{Notation.}
For elements $q \in L^2(\Omega)$, we denote the $L^2$-Norm by $\|q\| := \|q\|_{L^2(\Omega)}$. Furthermore $c$ is a generic constant, which may change from line to line, but is independent from the important variables, e.g. $k$.

\section{Assumptions and preliminary results}\label{sec:assumptions}
Let $\Omega \subseteq \R^n$, $n \in \N$ be a bounded, measurable domain, $Y$ a Hilbert space, $S: L^2(\Omega) \to Y$ linear and continuous. We are interested in the solution to problem \eqref{eq:main_problem}. Here we assume $z \in Y$ and $u_a, u_b \in L^\infty(\Omega)$ such that $u_a \leq u_b$. Hence the set of admissible controls {$\Uad$} is non-empty. By
\[
H(u) := \frac{1}{2}\|Su - z\|^{\mathbin{2}}
\]
we will denote our functional to be minimized.

\subsection{Existence of solutions}
Using classical arguments we can deduce existence of solutions.

\begin{theorem}
Under the assumptions listed above the problem \eqref{eq:main_problem} has a solution. If the operator $S$ is injective the solution is unique.
\end{theorem} 

Let $u^\dagger \in \Uad$ denote a solution of \eqref{eq:main_problem} with state $y^\dagger := Su^\dagger$ and adjoint state $p^\dagger := S^\ast(z-Su^\dagger)$. Note that due to the strict convexity of $H$ {with respect to $Su$} the optimal state $y^\dagger$ is uniquely defined. We now have the following result, see also \cite{wachsmuth2016}.

\begin{theorem}\label{thm:opt_udagger}
We have the relation {for almost all $ x\in \Omega$}
$$u^\dagger (x)  \begin{cases}
= u_a(x) & \text{if} \quad p^\dagger(x) < 0, \\
\in \left[u_a(x), u_b(x)\right] & \text{if} \quad p^\dagger(x) = 0, \\
= u_b & \text{if} \quad p^\dagger(x) > 0, \\
\end{cases}$$
and the following variational inequality holds:
$$(-p^\dagger, u - u^\dagger) \geq 0 \quad \forall u \in \Uad.$$
\end{theorem}

\subsection{Bregman iteration}
 In \cite{wachsmuth2016} we started to investigate an iterative method to solve \eqref{eq:main_problem} based on generalized Bregman distances. The Bregman distance \cite{Bregman67} $D^\lambda$ for a regularization functional $J$ at $u,v \in L^2(\Omega)$ is given by
$$D^\lambda(u,v) := J(u) - J(v) - (u-v, \lambda)$$
where $\lambda \in \partial J(v)$. We incorporate the control constraints into the regularization functional
$$J: L^2(\Omega) \to  \R \cup \{-\infty, +\infty\}, \quad J(u) := \frac{1}{2}\|u\|^2 + I_\Uad(u).$$
Let us recall some important properties of the regularization functional and the Bregman distance. {The next result can also be found in \cite[Lemma 2.3]{wachsmuth2016}.  }
\begin{lemma}
Let $C \subseteq L^2(\Omega)$ be non-empty, closed, and convex. The functional
$$J: L^2(\Omega) \to \mathbb{R}\cup \{ + \infty\}, \quad u \mapsto \frac{1}{2}\|u\|^2 + I_C(u)$$
is convex and nonnegative. Furthermore the Bregman distance
$$D^\lambda(u,v) := J(u) - J(v) - (u-v, \lambda), \quad \lambda \in \partial J(v)$$
is nonnegative and convex with respect to $u$.
\end{lemma}

{In the following we define $P_U$ to be the $L^2$-projection onto the set $U$.} Our algorithm is now given by: (see \cite{wachsmuth2016, burger2007})  

{
	\begin{myalg}\label{alg:MinEx}
		Let $u_0 = P_\Uad(0)  \in \Uad$, $\lambda_0=0  \in \partial J(u_0)$ and $k=1$.
		\begin{enumerate}
			\item Solve for $u_k$: \label{a0_start}
			\begin{equation*}
				\text{Minimize} \quad \frac{1}{2}\|Su-z\|_Y^2 + \alpha_{k}  D^{\lambda_{k-1}}(u,u_{k-1}).
			\end{equation*}
			\item Choose $\lambda_k \in \partial J(u_k)$.
			\item Set $k:=k+1$, go back to \ref{a0_start}.
		\end{enumerate}
	\end{myalg}
}

Here $(\alpha_k)_k$ is a bounded sequence of non-negative real numbers. In the next theorems we summarize some properties of the algorithm. The proofs can be found in \cite{wachsmuth2016}. In the following we use the abbreviation
$$\gamma_k := \sum\limits_{j=1}^k \frac{1}{\alpha_j}.$$

Let us first recall a convergence result in terms of the functional $H$.
\begin{theorem}
Algorithm \ref{alg:MinEx} is well-posed and we have $\lambda_k \in \partial J(u_k)$ for all $k \in \N_0$. Let $u^\dagger$ be a solution of \eqref{eq:main_problem}. We then have
\begin{align*}
H(u_k) &\leq H(u_{k-1}),\\
|H(u_k) - H(u^\dagger)| &= \mathcal{O}\left( \gamma_k^{-1} \right).
\end{align*}
Furthermore we have the monotonicity property of the sequence $(u_k)_k$ with respect to the Bregman distance
$$D^{\lambda_k}(u^\dagger , u_k) \leq D^{\lambda_{k-1}}(u^\dagger , u_{k-1})$$
and
$$\sum\limits_{i=1}^\infty D^{\lambda_{i-1}}(u_i,u_{i-1}) < \infty.$$
\end{theorem}

We also established a general convergence result in terms of the controls.
\begin{theorem}
Weak limit points of the sequence $(u_k)_k$ generated by Algorithm \ref{alg:MinEx} are solutions to the problem \eqref{eq:main_problem}. Furthermore we obtain strong convergence of the states
$$Su_k \to y^\dagger,$$
where $y^\dagger$ is the uniquely determined optimal state of \eqref{eq:main_problem}. If in addition $u^\dagger$ is the unique solution of \eqref{eq:main_problem}, we obtain
$u_k \to u^\dagger.$
\end{theorem}

In order to establish convergence rates for the iterates of Algorithm \ref{alg:MinEx} we have to assume some regularity on the solution of \eqref{eq:main_problem}. A common assumption on a solution $u^\dagger$ is the following source condition, which is an abstract smoothness condition, see, e.g.,  \cite{burger2007,chaventkunisch94,itojin11,neubauer1988,wachsmuth2011,wachsmuth2011b}. We say $u^\dagger$ satisfies the source condition \ref{ass:SC} if the following assumption holds.

{
\renewcommand{\thetheorem}{\textbf{SC}}
\begin{assumption}[Source Condition]\label{ass:SC}
Let $u^\dagger$ be a solution of \eqref{eq:main_problem}.
Assume that there exists an element $w \in Y$ such that $u^\dagger = P_\Uad(S^\ast w)$ holds.
\end{assumption}
}

This assumption is too restrictive as in many cases the solution $u^\dagger$ is bang-bang, i.e. a linear combination of characteristic functions, hence discontinuous. But in many applications the range of $S^\ast$ contains $H^1(\Omega)$ or $C(\bar \Omega)$, hence the Assumption \ref{ass:SC} is not applicable in this case. To overcome this, we use the regularity of the adjoint state.  We say $u^\dagger$ satisfies the source condition \ref{ass:ActiveSet} if the following assumption holds. {In the following we define $\chi_A$ to be the indicator function of the set $A$.} Recall that the adjoint state is defined by $p^\dagger = S^\ast(z-Su^\dagger)$.

{
\renewcommand{\thetheorem}{\textbf{ASC}}
\begin{assumption}[Active Set Condition]\label{ass:ActiveSet}
Let $u^\dagger$ be a solution of \eqref{eq:main_problem} and assume that there exists a set $I \subseteq \Omega$, a function $w \in Y$, and positive constants $\kappa, c$ such that the following holds

\begin{enumerate}
  \item (source condition) $I \supset \{ x \in \Omega: \; p^\dagger(x) = 0 \}$  and
  $$\chi_I u^\dagger = \chi_I P_\Uad (S^\ast w),$$
  \item (structure of active set) $A := \Omega \setminus I$ and for all $\eps > 0$
  $$|\{ x\in A: \; 0 < |p^\dagger(x)| < \eps  \}| \leq c \eps^\kappa,$$
  \item (regularity of solution) $S^\ast w \in L^\infty(\Omega)$.
\end{enumerate}
\end{assumption}
}

Assumption \ref{ass:ActiveSet} is a generalization of Assumption \ref{ass:SC}, since for $I = \Omega$ both assumptions coincide. A sufficient condition for Assumption \ref{ass:ActiveSet} can be found in \cite{hinze2012}. If $p^\dagger \in C^1(\bar \Omega)$ satisfies
$$\nabla p^\dagger \neq 0 \quad \forall x \in \bar \Omega \quad \text{with} \quad p^\dagger(x) = 0$$
Assumption \ref{ass:ActiveSet} is fulfilled with $A = \Omega$ and $\kappa = 1$. Since Assumption \ref{ass:SC} omits more regularity, we expect to establish improved results in this case. The regularity assumption \ref{ass:ActiveSet} is used in e.g. \cite{wachsmuth2011,wachsmuth2011b,wachsmuth2013,wachsmuth2016}.

Using this regularity assumptions we established in \cite{wachsmuth2016} the following convergence results.

\begin{theorem}\label{thm:SC_strong_conv}
Let $(u_k)_k$ be the sequence generated by Algorithm \ref{alg:MinEx}. Assume that Assumption \ref{ass:SC} holds for $u^\dagger$. Then
\begin{align*}
\|u^\dagger - u_k\|^2 &= \mathcal{O}(\gamma_k^{-1}) \quad \text{and} \quad \sum\limits_{i=1}^k \frac{1}{\alpha_i} \|u^\dagger - u_i\|^2 \leq c.
\end{align*}
If we assume that instead Assumption \ref{ass:ActiveSet} holds, then
\begin{align*}
\|u^\dagger - u_k\|^2 &= \mathcal{O}\left( \gamma_k^{-1} + \gamma_k^{-1} \sum\limits_{j=1}^k \alpha_j^{-1} \gamma_j^{- \kappa}   \right)\\
\text{and} \quad \sum\limits_{i=1}^k \frac{1}{\alpha_i} \|u^\dagger - u_i\|^2 &\leq c \left( 1 + \sum\limits_{i=1}^k \alpha_i^{-1} \gamma_i^{-\kappa} \right).
\end{align*}
\end{theorem}
Note that we have 
$$\gamma_k^{-1} + \gamma_k^{-1} \sum\limits_{j=1}^k \alpha_j^{-1} \gamma_j^{- \kappa} \to 0  \quad\text{as} \quad k \to \infty,$$
see \cite{wachsmuth2016}. If Assumption \ref{ass:ActiveSet} holds with $A = \Omega$, which implies that $u^\dagger$ is bang-bang on $\Omega$, we can improve the estimate of Theorem \ref{thm:SC_strong_conv} to
$$\|u^\dagger - u_k\|^2 = \mathcal{O}\left( \gamma_k^{-1} \sum\limits_{j=1}^k \alpha_j^{-1} \gamma_j^{- \kappa} \right).$$
The sequence of subdifferentials $(\lambda_k)_k$ is unbounded in general, but we can show that the weighted average $(\gamma_k^{-1}\lambda_k)_k$ is converging. {The proof can be found in \cite[Corollary 4.14]{wachsmuth2016}. }

\begin{lemma}
We have
\[
\left\|\gamma_k^{-1} \lambda_k - p^\dagger \right\|^2 =  \begin{cases}
    \mathcal O(\gamma_k^{-2}) & \text{if $u^\dagger$ satisfies \ref{ass:SC}}, \\
    \mathcal O\left(\gamma_k^{-2}\left( 1   +  \sum\limits_{j=1}^k \alpha_j^{-1} \gamma_j^{-\kappa}\right)\right) & \text{if $u^\dagger$ satisfies \ref{ass:ActiveSet}.}
    \end{cases} \]
\end{lemma}

{
	
	\section{The discretized problem} \label{section:descr_problem}
	The aim of this section is to introduce the operator $S_h$ and to establish auxiliary estimates for the discretized subproblem. These estimates will then be applied to prove convergence results in section \ref{sec:inexact_bregman}.
	
	\subsection{The operator $S_h$}
	As mentioned in the introduction we want to introduce a family of linear and continuous operators $(S_h)_h$ from $L^2(\Omega)$ to $Y$ with finite-dimensional range $Y_h \subset Y$. Throughout this paper we make the following assumption. A similar assumption is also made in \cite{wachsmuth2013adaptive}.
	
	\begin{assumption}\label{ass:operator_Sh}
	Assume that there exists a continuous and monotonically increasing function $\delta: \R^+ \to \R^+$ with $\delta(0) = 0$ such that 
		\[
		\|(S-S_h) u_h\|_Y + \|(S^\ast - S_h^\ast)(y_h - z )\| \leq \delta(h)
		\]
	holds for all $h \geq 0$, $u_h \in \Uad$ and $y_h := S_h u_h$.
	\end{assumption}
	
	For the case of a linear elliptic partial differential equation, the operator $S_h$ is the solution operator of the weak formulation with respect to the test function space $Y_h$.  If $Y_h$ is spanned by linear finite elements, this can be interpreted as the variational discretization in the sense of Hinze, see \cite{Hinze2005}. We consider a linear elliptic partial differential equation in section \ref{sec:numerics}. We assume that the operator $S_h$ and its adjoint $S_h^\ast$ can be computed exactly.
	
	Note that \ref{ass:operator_Sh} is an assumption on the approximation of discrete functions. Under Assumption \ref{ass:operator_Sh} we can establish the following discretization error estimate. The proof is similar to \cite[Proposition 1.6]{wachsmuth2013adaptive} and is omitted here.
	\begin{lemma}\label{lemma:subprob_ex_error}
		Let $u_{k}$ be the solution of 
		\[
		\min\limits_{u \in \Uad} \frac{1}{2}\|Su - z\|^2_Y +  \frac{\alpha_k}{2} \|u\|^2 - \alpha_k(\lambda,u),
		\]
		and $u_{k,h}$ be the solution of the discretized problem
		\[
		\min\limits_{u \in \Uad} \frac{1}{2}\|S_h u - z\|^2_Y +  \frac{\alpha_k}{2} \|u\|^2 - \alpha_k(\lambda,u),
		\]
		with $\lambda \in L^2(\Omega)$ and $\alpha > 0$. Then we have the following estimate
		\[
		\frac{1}{\alpha_{k}} \| y_{k,h} - y_{k}  \|_Y^2 + \| u_{k,h} - u_{k} \|^2 \leq c  \rho^2_{k}\delta (h)^2
		\]
		with the abbreviation $\rho^2_{k} := \alpha_{k}^{-1} (1 + \alpha_{k}^{-1}) $.
	\end{lemma}

	Please note that the norm of the operator $S_h$ is bounded in the following sense.
	\begin{lemma}\label{lemma:norm_Sh_bounded}
	Let $0 < h \leq h_{\max}$. Then there exists a constant $C > 0$ independent from $h$, such that $\|S_h\| \leq C$.
	\end{lemma}
	\begin{proof}
	We compute the operator norm of $S_h$ and estimate
	\begin{align*}
	\|S_h\| &= \sup\limits_{\|u\|=1} \|S_h u\|_Y \leq \sup\limits_{\|u\|=1} (  \|(S_h - S) u\|_Y + \|Su\|_Y )\\
	&\leq \delta(h) + \|S\|\\
	&\leq \delta(h_{\max}) + \|S\|.
	\end{align*}
	\end{proof}

	In the subsequent analysis we will need the following estimate.
	\begin{lemma}
	There exists a constant $c>0$ independent from $h$, such that the following estimate holds for all $u_h \in \Uad$
		\[
		\|  S_h^\ast (z - S_h u_h) - S^\ast(  z - S u_h )  \| \leq c \delta (h).
		\]
	\end{lemma}
	
	\begin{proof}
		We compute with $y_h := S_h u_h$
		\begin{align*}
			\|S_h^\ast (z - S_h u_h)& - S^\ast(  z - S u_h ) \|\\
			&\leq \| S_h^\ast(z - S_h u_h) - S^\ast (z - S_h u_h) \|+ \|S^\ast (S u_h - S_h u_h)\|\\
			&\leq \|(S_h^\ast - S^\ast) (z - y_h)\| + c \|(S-S_h) u_h\|_Y\\
			&\leq c \delta(h).
		\end{align*}
		Please note that we used the continuity of $S^\ast$ and the assumption on the operator $S_h$.
	\end{proof}
	
	As a corollary we obtain the following result.
	\begin{lemma}\label{lemma:estimate_lambda_different_meshes}
		Let $u_i \in \Uad$ for $i=1,..,k$. Then there exists a constant $c>0$ independent from $h$ and $k$ such that the following estimate holds 
		\[
		\left\|  \sum\limits_{i=1}^k \frac{1}{\alpha_i} S^\ast (z - S u_i) - \sum\limits_{i=1}^k \frac{1}{\alpha_i} S_h^\ast(z - S_h u_i)  \right\| \leq c \gamma_k \delta(h).
		\]
	\end{lemma}
	
}

\subsection{A-posteriori error estimate for the {discretized} subproblem}
{We now want to consider the discretized subproblem, i.e. we replaced the operator $S$ in the minimization problem (step 1) of algorithm \ref{alg:MinEx} with the discrete operator $S_h$. This gives the following problem}
\begin{equation*}
\text{Minimize} \quad \frac{1}{2}\| {S_h} u-z\|_Y^2 + \alpha_{k}  D^{\lambda_{k-1}}(u,u_{k-1}).
\end{equation*}
This problem can be rewritten as the equivalent minimization problem \eqref{eq:subprob_abstract}, see also \cite{wachsmuth2016}. For brevity we set $\lambda := \lambda_{k-1}$ and $\alpha := \alpha_k$.

\begin{equation}\label{eq:subprob_abstract}
\begin{split}
\text{Minimize} \quad &\frac{1}{2}\|{S_h}u-z\|_Y^2 - \alpha (\lambda, u) + \frac{\alpha}{2}\|u\|^2,\\
\text{s.t.} \quad & u \in \Uad.
\end{split}
\end{equation}

To construct an a-posteriori error estimate we use Theorem 2.2 in \cite{roesch2014}, which will give us the following result.  Note that we also use Lemma \ref{lemma:norm_Sh_bounded} here.
{
\begin{theorem}\label{thm:apostestimate}
Let $\hat u$ be the solution of the subproblem \eqref{eq:subprob_abstract}.  Let $u_h \in L^2(\Omega)$ be given and define $y_h := S_h u_h$ and $p_h := S_h^\ast (z - y_h)$. Let $0 < h \leq h_{\max}$ with $h_{\max} > 0$. Then there exists a constant $c > 0$ independent from $h$ such that
\begin{align*}
\|u_h - \hat u\| \leq \; &  c \left( 1+ \frac{1}{\alpha} \right) \left\| u_h - P_\Uad \left( \frac{1}{\alpha}p_h  + \lambda \right) \right\|.
\end{align*}
\end{theorem}
}

This results allows us to estimate the distance to the exact solution of the subproblem. Note that the problem \eqref{eq:subprob_abstract} is uniquely solvable if $\alpha > 0$, see \cite{wachsmuth2016}.  

For abbreviation we set
{
\begin{align*}
\mathcal{B}(\alpha,\lambda,u_h)  :=&   \left(  1 + \frac{1}{\alpha} \right)  \left\| u_h - P_\Uad \left( \frac{1}{\alpha}S_h^\ast(z - S_h u_h )  + \lambda \right) \right\|.
\end{align*}

Let $u \in L^2(\Omega)$ be an approximate solution to the discretized subproblem  \eqref{eq:subprob_abstract}. The quantity $\mathcal{B}(\alpha, \lambda, u)$ then is an upper bound for the accuracy of $u$. This is part of the next result.  The proof follows directly with Lemma \ref{lemma:norm_Sh_bounded} and Theorem \ref{thm:apostestimate}.
\begin{lemma}\label{lemma:a_post_est_subprob}
Assume that $0 < h \leq h_{\max}$. Let $\hat u$ be the solution of the discretized subproblem \eqref{eq:subprob_abstract}. Then there exists a constant $c > 0$ independent from $h$ such that the following implication holds for all $u \in L^2(\Omega)$ and $\eps \geq 0$:
\[
 \mathcal{B}(\alpha, \lambda, u) \leq \eps \implies \|u - \hat u\| \leq c \eps.
\]
\end{lemma}

Let us close this section with the following remark. As mentioned in \cite{Hinze2005} the solution of the discretized subproblem \eqref{eq:subprob_abstract} can be approximated with arbitrary accuracy. This will play a role in the analysis presented in the next section.

}

\section{Inexact Bregman iteration}
\label{sec:inexact_bregman}
Solving the subproblem
$$\text{Minimize} \quad \frac{1}{2}\|Su-z\|_Y^2 + \alpha_{k}  D^{\lambda_{k-1}}(u,u_{k-1})$$
exactly is very costly and in general not possible. We therefore suggest the following inexact Bregman iteration which can be interpreted as an inexact version of Algorithm \ref{alg:MinEx}.

Inexact Bregman iterations are analysed in the literature, see e.g. \cite{eckstein98, tichatschke2004,Kiwiel1997,Benfenati2013} for a finite dimensional approach, and for an abstract Banach space setting, see \cite{Sabach2010b}.

{
	\begin{myalg}\label{alg:MinInEx}
		Let $u_0^\tin = P_\Uad(0)  \in \Uad$, $\lambda_0^\tin=0  \in \partial J(u_0)$ and $k=1$.
		\begin{enumerate}
			\item  Find $u_k^\tin$ with $y_k^\tin = {S_h} u_k^\tin$ and $p_k^\tin = {S_h}^\ast(z - {S_h}u_k^\tin)$ such that 
			\[\mathcal{B}(\alpha_k, \lambda_{k-1}^\tin, u_k^\tin) \leq \eps_k\]
			\item Set
			\[\lambda_k^\tin = \sum\limits_{i=1}^k \frac{1}{\alpha_i} S_h^\ast (z-{S_h}u_i^\tin)\]
			\item Set $k:=k+1$, go back to 1.
		\end{enumerate}
	\end{myalg}
}

Here $\eps_k \geq 0$ is a given sequence of positive real numbers controlling the accuracy of the approximate solution $u_k^\tin$. For $\eps_k = 0$ {for all $k \in \N$ and $h=0$} Algorithm \ref{alg:MinEx} is obtained. \\

{The analysis of Algorithm \ref{alg:MinEx} presented in \cite{wachsmuth2016} is based on the fact that $\lambda_k \in \partial J(u_k)$. This is guaranteed by the construction of $\lambda_k$. However, since $S_h \neq S$ and $\eps_k > 0$ in general, we cannot expect that   $\lambda_k^\tin \not\in \partial J(u_k^\tin)$ holds.\\
	
Before we start to establish robustness results we want to give an overview over the different auxiliary problems we are going to use. Furthermore we want to introduce and clarify our notation.}

{
\subsection{Notation and auxiliary results}
The aim of this section is to summarize the most important notations and abbreviations. Our aim is to solve the \emph{unregularized problem}
\[
\min\limits_{u \in \Uad} \quad \frac{1}{2} \|Su - z\|_Y^2.
\]
This problem is solvable and we want to specify a solution $u^\dagger$. We assume, that this function satisfies one of the regularity assumptions \ref{ass:SC} or \ref{ass:ActiveSet}. In Algorithm \ref{alg:MinEx} we have to solve the following \emph{regularized problem}. We will refer to this as subproblem
\begin{equation}\label{eq:not_1}
\min\limits_{u \in \Uad} \quad \frac{1}{2} \|Su - z\|_Y^2 + \frac{\alpha_{k+1}}{2}\|u\|^2 - \alpha_{k+1}(\lambda, u),
\end{equation}
with some $\lambda \in L^2(\Omega)$ and $\alpha_{k+1} > 0$. Here the (exact) unique solution is denoted with $u_{k+1}^\tex$.  The superscript \emph{ex} stands for \emph{exact solution}.

However, since the operator $S$ is not computable in general, we introduced the operator $S_h$, which is an approximation of $S$. We now replace $S$ with $S_h$ in \eqref{eq:not_1} and obtain the \emph{discretized subproblem}
\begin{equation*}
\min\limits_{u \in \Uad} \quad \frac{1}{2} \|S_hu - z\|_Y^2 + \frac{\alpha_{k+1}}{2}\|u\|^2 - \alpha_{k+1}(\lambda, u).
\end{equation*}
Again this problem is unique solvable and its solution is denoted with $u_{k+1,h}^\tex$. The subscript $h$ indicates that it is a \emph{discrete} solution. Under suitable assumptions we can estimate the discretization error between $u_{k+1}^\tex$ and $u_{k+1,h}^\tex$. This is done in Theorem \ref{lemma:subprob_ex_error}.

 Please note that neither $u_{k+1}^\tex$ nor $u_{k+1,h}^\tex$ are computed during the algorithm. As mentioned above we can approximate $u_{k+1,h}^\tex$ with arbitrary precision. So we compute an \emph{inexact} solution of \eqref{eq:subprob_abstract}, which is denoted with $u_{k+1}^\tin$. We use the function $\mathcal{B}$ to measure the accuracy.

To control the accuracy during the algorithm we introduce a sequence $(\eps_k)_k$ of positive real values. In each iteration we now search for a function $u_{k+1}^\tin \in \Uad$ such that $\mathcal{B}(\alpha, \lambda, u_{k+1}^\tin) \leq \eps_{k+1}$.

In the end we want to estimate the error $\|u^\dagger - u_{k}^\tin\|$. This is done by triangular inequality
\begin{equation}\label{eq:triangular}
\|u_k^\tin - u^\dagger\| \leq \overbrace{\|u_k^\tin - u_{k,h}^\tex\|}^{(I)} + \overbrace{\|u_{k,h}^\tex - u_k^\tex\|}^{(II)} + \overbrace{\|u_k^\tex - u^\dagger\|}^{(III)}.
\end{equation}
Note that $(I)$ is controlled by the accuracy $\eps_k$ and $(II)$ is limited by the discretization error. It remains to estimate the regularization error $(III)$ with the help of the regularity assumptions.

We also want to recall the following definitions, as they will appear quite often.
\[
\gamma_k = \sum\limits_{i=1}^k \frac{1}{\alpha_i}, \quad \rho_k^2 = \alpha_k^{-1}(1 + \alpha_k^{-1}).
\]

}

\subsection{Convergence under Assumption \ref{ass:SC}}

We now start to analyse Algorithm \ref{alg:MinInEx} with $u^\dagger$ satisfying Assumption \ref{ass:SC}.

{

\begin{theorem}\label{thm:convergence_SC}
Let $u^\dagger$ satisfy Assumption \ref{ass:SC} and let $(\eps_k)_k$ be a sequence of positive real numbers. Furthermore let $h > 0$ be given and let $(u_k^\tin)_k$ be a sequence generated by Algorithm \ref{alg:MinInEx}. Then we have the estimate
\begin{align*}
\sum\limits_{i=1}^k  \frac{1}{\alpha_{i}}\|u_{i}^\tin - u^\dagger\|^2 \leq c\left( 1 +  \sum\limits_{i=1}^k R_i +  \sum\limits_{i=1}^k H_i \right)
\end{align*}

with the abbreviations
\begin{align*}
R_i &:= \frac{\eps_{i}}{\alpha_{i}} + \frac{\eps_{i}^2}{\alpha_{i}^2}+ \frac{\gamma_{i-1} \eps_{i}}{\alpha_{i}} + \frac{\eps_i^2}{\alpha_i},\\
H_i &:= \delta(h) \left[ \frac{\rho_{i}}{\alpha_i} + \frac{\gamma_{i-1}}{\alpha_i} + \frac{\gamma_{i-1} \rho_i}{\alpha_i} \right] + \delta(h)^2 \left[ \frac{\rho_i^2}{\alpha_i^2} +\frac{\rho_i^2}{\alpha_i} \right].
\end{align*}

\end{theorem}

\begin{proof}
The proof is based on the splitting of the error $\|u_k^\tin - u^\dagger\|$  in three parts, see \eqref{eq:triangular}
\[
\|u_k^\tin - u^\dagger\| \leq \overbrace{\|u_k^\tin - u_{k,h}^\tex\|}^{(I)} + \overbrace{\|u_{k,h}^\tex - u_k^\tex\|}^{(II)} + \overbrace{\|u_k^\tex - u^\dagger\|}^{(III)}.
\]
Here $(I)$ is controlled by the given accuracy $\eps_k$ and $(II)$ can be estimated with the help of Lemma \ref{lemma:subprob_ex_error}:
\begin{align*}
(I) = \|u_k^\tin - u_{k,h}^\tex\| &\leq c\eps_k,\\
(II) = \|u_{k,h}^\tex - u_k^\tex\| &\leq c\rho_{k} \delta(h).
\end{align*}

It is left to estimate $(III)$. We start with adding the optimality conditions for $u_{k+1}^\tex$ and $u^\dagger$, see \cite[Lemma 3.1]{wachsmuth2016} and Theorem \ref{thm:opt_udagger},
\begin{align*}
\big( S^\ast(Su_{k+1}^\tex - z) + \alpha_{k+1}( u_{k+1}^\tex - \lambda_k^\tin ), v - u_{k+1}^\tex  \big) & \geq 0, \quad \forall v \in \Uad,\\
\big( S^\ast(Su^\dagger - z), v - u^\dagger \big) & \geq 0, \quad \forall v \in \Uad.
\end{align*}
Addition yields
\begin{equation}\label{eq:thminex_1}
\frac{1}{\alpha_{k+1}} \|S (u_{k+1}^\tex - u^\dagger)\|_Y^2 + \|u_{k+1}^\tex - u^\dagger\|^2 \leq \big( u^\dagger - \lambda_k^\tin , u^\dagger - u_{k+1}^\tex \big).
\end{equation}
For the term $(u^\dagger, u^\dagger - u_{k+1}^\tex)$ we estimate with help of the source condition \ref{ass:SC}
\begin{equation}\label{eq:thminex_3}
\begin{split}
(u^\dagger, u^\dagger - u_{k+1}^\tex) &= (u^\dagger, u^\dagger - u_{k+1}^\tin)\\
&\quad+ (u^\dagger, u_{k+1}^\tin - u_{k+1,h}^\tex) + (u^\dagger, u_{k+1,h}^\tex - u_{k+1}^\tex)\\
&\leq (S^\ast w, u^\dagger - u_{k+1}^\tin) + c( \eps_{k+1} + \rho_{k+1} \delta(h) ).
\end{split}
\end{equation}
To estimate the remaining term $(- \lambda_k^\tin , u^\dagger - u_{k+1}^\tex \big))$ we introduce the quantity
\[
v_k^\tin := \sum\limits_{i=1}^k \frac{1}{\alpha_i} S(u^\dagger - u_i^\tin).
\]
This quantity will be helpful in the subsequent analysis. Let us sketch the next steps. First we will replace the operator $S_h$ by $S$ in order to apply the first order conditions for $u^\dagger$. Second we eliminate the unknown exact solution $u_{k+1}^\tex$ by its approximation $u_{k+1}^\tin$. For the first part we make use of Lemma \ref{lemma:estimate_lambda_different_meshes} and estimate
\begin{equation}\label{eq:thminex_0}
\begin{split}
(-\lambda_k^\tin, u^\dagger &- u_{k+1}^\tex) = \left( \sum\limits_{i=1}^k  \frac{1}{\alpha_i}  S_h^\ast ( S_h u_i^\tin - z ), u^\dagger - u_{k+1}^\tex   \right)\\
&=  \left( \sum\limits_{i=1}^k  \frac{1}{\alpha_i}  S^\ast ( S u_i^\tin - z ), u^\dagger - u_{k+1}^\tex   \right)\\
&+\quad \left( \sum\limits_{i=1}^k  \frac{1}{\alpha_i}  S_h^\ast ( S_h u_i^\tin - z ) - \sum\limits_{i=1}^k  \frac{1}{\alpha_i}  S^\ast ( S u_i^\tin - z ), u^\dagger - u_{k+1}^\tex   \right)\\
&=  \left( \sum\limits_{i=1}^k  \frac{1}{\alpha_i}  S^\ast ( S u_i^\tin - z ), u^\dagger - u_{k+1}^\tex   \right) + c \gamma_k \delta(h).
\end{split}
\end{equation}
Now we eliminate the variable $z$ by using the first order conditions for $u^\dagger$ presented in Theorem \ref{thm:opt_udagger}
\begin{equation}\label{eq:eliminate_z}
\begin{split}
&\left(\sum\limits_{i=1}^k \frac{1}{\alpha_i} ( Su_i^\tin -z, S(u^\dagger - u_{k+1}^\tex) \right)\\
&= \sum\limits_{i=1}^k \frac{1}{\alpha_i} (Su_i^\tin - Su^\dagger, S(u^\dagger - u_{k+1}^\tex)) + \sum\limits_{i=1}^k \frac{1}{\alpha_i} \underbrace{(Su^\dagger - z, S(u^\dagger - u_{k+1}^\tex))}_{\leq 0}\\
&\leq \sum\limits_{i=1}^k \frac{1}{\alpha_i} \big( S(u_i^\tin - u^\dagger), S(u^\dagger - u_{k+1}^\tex) \big).
\end{split}
\end{equation}
Since the variable $u_{k+1}^\tex$ is unknown we replace it by its approximation $u_{k+1}^\tin$
\begin{equation}\label{eq:thminex_2}
\begin{split}
\sum\limits_{i=1}^k \frac{1}{\alpha_i} &\big( S(u_i^\tin - u^\dagger), S(u^\dagger - u_{k+1}^\tex) \big)\\
&= \sum\limits_{i=1}^k \frac{1}{\alpha_i} \big( S(u_i^\tin - u^\dagger), S(u^\dagger - u_{k+1}^\tin) \big)\\
&\quad+ \sum\limits_{i=1}^k \frac{1}{\alpha_i} \big( S(u_i^\tin - u^\dagger), S(u_{k+1}^\tin - u_{k+1,h}^\tex) \big)\\
&\quad+ \sum\limits_{i=1}^k \frac{1}{\alpha_i} \big( S(u_i^\tin - u^\dagger), S(u_{k+1,h}^\tex - u_{k+1}^\tex) \big)\\
& \leq \alpha_{k+1} ( -v_k^\tin , v_{k+1}^\tin - v_k^\tin ) + c \gamma_k(\eps_{k+1} + \rho_{k+1} \delta(h)) .
\end{split}
\end{equation}
Now we use \eqref{eq:eliminate_z} and \eqref{eq:thminex_2} in \eqref{eq:thminex_0} and obtain
\begin{equation}\label{eq:thminex_6}
\begin{split}
(- \lambda_k^\tin, u^\dagger - u_{k+1}^\tex) &\leq \alpha_{k+1} (-v_k^\tin, v_{k+1}^\tin - v_k^\tin)\\
&\quad+ c \big (\gamma_k \eps_{k+1} + \delta(h) \gamma_k \rho_{k+1} + \delta(h) \gamma_k  \big).
\end{split}
\end{equation}

In the next step we plug \eqref{eq:thminex_3} and \eqref{eq:thminex_6} in \eqref{eq:thminex_1}
\begin{align*}
\frac{1}{\alpha_{k+1}^2} &\|S(u_{k+1}^\tex - u^\dagger)\|_Y^2 + \frac{1}{\alpha_{k+1}}\|u_{k+1}^\tex - u^\dagger\|^2\\
&\leq \frac{1}{\alpha_{k+1}} (u^\dagger , u^\dagger - u_{k+1}^\tex) + \frac{1}{\alpha_{k+1}}(- \lambda_k^\tin, u^\dagger - u_{k+1}^\tex)\\
&\leq  \left( w, \frac{1}{\alpha_{k+1}}S( u^\dagger - u_{k+1}^\tin )  \right) +  (-v_k^\tin, v_{k+1}^\tin - v_k^\tin)\\
&\quad  + c \left( \frac{\eps_{k+1}}{\alpha_{k+1}}  + \frac{\gamma_k \eps_{k+1}}{\alpha_{k+1}} \right) +  c \delta(h) \left( \frac{\rho_{k+1}}{\alpha_{k+1}} + \frac{\gamma_k}{\alpha_{k+1}} + \frac{\gamma_k \rho_{k+1}}{\alpha_{k+1}} \right)\\
&\leq  (w-v_k^\tin, v_{k+1}^\tin - v_k^\tin)\\
&\quad  + c \left( \frac{\eps_{k+1}}{\alpha_{k+1}}  + \frac{\gamma_k \eps_{k+1}}{\alpha_{k+1}} \right) +  c \delta(h) \left( \frac{\rho_{k+1}}{\alpha_{k+1}} + \frac{\gamma_k}{\alpha_{k+1}} + \frac{\gamma_k \rho_{k+1}}{\alpha_{k+1}} \right).
\end{align*}

Before we proceed we need two additional results. A calculation reveals that
\begin{equation}\label{eq:thminex_4}
(w-v_k^\tin, v_{k+1}^\tin - v_k^\tin) = \frac{1}{2}\|v_k^\tin-w\|_Y^2 - \frac{1}{2}\|v_{k+1}^\tin-w\|_Y^2 + \frac{1}{2}\|v_{k+1}^\tin - v_k^\tin\|_Y^2
\end{equation}
holds.  Second we obtain
\begin{equation}\label{eq:tech_est}
\begin{split}
\|S(u_{k+1}^\tex - u_{k+1}^\tin)\|_Y &\leq \|S(u_{k+1}^\tin - u_{k+1,h}^\tex)\|_Y + \|S(u_{k+1,h}^\tex - u_{k+1}^\tex)\|_Y\\
&\leq c( \eps_{k+1} + \rho_{k+1} \delta(h) ).
\end{split}
\end{equation}

Furthermore we use Young's inequality and \eqref{eq:tech_est} to establish for $\tau > 1$:
\begin{equation}\label{eq:thminex_5}
\begin{split}
&\frac{1}{2}\|v_{k+1}^\tin - v_k^\tin\|^2_Y = \frac{1}{2 \alpha_{k+1}^2} \|S(u_{k+1}^\tin - u^\dagger)\|_Y^2 \\
&\leq \frac{1}{2 \alpha_{k+1}^2} \left(  \left(1 + \frac{1}{\tau}\right) \|S(u_{k+1}^\tex - u^\dagger)\|_Y^2 + \left(1 + \tau\right)\|S(u_{k+1}^\tex - u_{k+1}^\tin)\|_Y^2 \right)\\
& \leq \frac{1}{2\alpha_{k+1}^2} \left(1 + \frac{1}{\tau}\right) \|S (u_{k+1}^\tex - u^\dagger)\|_Y^2 + \frac{c }{\alpha_{k+1}^2} \left(\eps_{k+1}^2 + \rho_{k+1}^2 \delta(h)^2 \right).
\end{split}
\end{equation}
This now yields
\begin{align*}
\frac{c_\tau}{\alpha_{k+1}^2} &\|S(u_{k+1}^\tex - u^\dagger)\|_Y^2  + \frac{1}{\alpha_{k+1}}\|u_{k+1}^\tex - u^\dagger\|^2 \\
&\leq \frac{1}{2}\|v_k^\tin - w\|_Y^2 - \frac{1}{2}\|v_{k+1}^\tin - w\|_Y^2 +c \left(\tilde R_{k+1} + \delta(h) \tilde H^{(1)}_{k+1} + \delta(h)^2 \tilde H^{(2)}_{k+1} \right),
\end{align*}
with $c_\tau = 1 - \frac{1}{2}\left(1 + \frac{1}{\tau}\right) > 0$ and the abbreviations
\begin{align*}
\tilde R_i &:=   \frac{\eps_{i}}{\alpha_{i}} + \frac{\eps_{i}^2}{\alpha_{i}^2}+ \frac{\gamma_{i-1} \eps_{i}}{\alpha_{i}},\\
\tilde H_i^{(1)} &:= \frac{\rho_{i}}{\alpha_i} + \frac{\gamma_{i-1}}{\alpha_i} + \frac{\gamma_{i-1} \rho_i}{\alpha_i},\\
\tilde H_i^{(2)} &:=  \frac{\rho_i^2}{\alpha_i^2}.
\end{align*}

  Summation over $k$ finally reveals
\begin{align*}
c_\tau \sum\limits_{i=1}^k \frac{1}{\alpha_{i}^2} &\|S(u_{i}^\tex - u^\dagger)\|_Y^2  + \sum\limits_{i=1}^k  \frac{1}{\alpha_{i}}\|u_{i}^\tex - u^\dagger\|^2 \\
&\leq \frac{1}{2}\|w\|_Y^2  +c \sum\limits_{i=1}^k \tilde R_{i} + c \delta(h) \sum\limits_{i=1}^k \tilde H^{(1)}_{i} + c \delta(h)^2 \sum\limits_{i=1}^k  \tilde H^{(2)}_{i} ,
\end{align*}
where we used the convention $v_0^\tin = 0$. The result now follows by triangular inequality.

\end{proof}

Let us point out that the variables $R_i$ can be identified with the accuracy of the iterates and while the $H_i$ are only influenced by the discretization. This result above can now be interpreted in different ways.  First we start with the (theoretical) case that we can evaluate the operator $S$ and its dual $S^\ast$. This refers to the case where $h=0$.
\begin{corollary}\label{cor:SC_1}
Let $u^\dagger$ satisfy Assumption \ref{ass:SC} and let $(\eps_k)_k$ be a sequence of positive real numbers such that
\[
\sum\limits_{i=1}^\infty R_i < \infty.
\]
Furthermore assume that $S_h = S$ and let $(u_k^\tin)_k$ be a sequence generated by Algorithm \ref{alg:MinInEx}. Then we have $u_k^\tin \to u^\dagger$ in $L^2(\Omega)$.
\end{corollary}

The other interesting case is, that we can solve the discretized subproblem exactly, i.e. $\eps_k = 0$ for all $k \in \N$. Here we obtain convergence in the following sense.

\begin{corollary}\label{cor:SC_2}
	Let $u^\dagger$ satisfy Assumption \ref{ass:SC}. Let $h_{\max} > 0$ be given and $\eps_k = 0$ for all $k \in \N$. Then there exists a constant C such that  for every $0 < h \leq h_{\max}$ there exists a stopping index $k(h)$ such that
	\[
	\sum\limits_{i=1}^{k(h)} H_i \leq C < \infty.
	\]
	and $k(h) \to \infty$ as $h \to 0$. Furthermore $u_{k(h)}^\tin \to u^\dagger$ as $h \to 0$.
\end{corollary}

\begin{proof}
	We only have to show the existence of such a stopping index. The convergence result then is a direct consequence of Theorem \ref{thm:convergence_SC}. Let us define the following auxiliary variables
	\begin{align*}
	A_i &:=\frac{\rho_{i}}{\alpha_i} + \frac{\gamma_{i-1}}{\alpha_i} + \frac{\gamma_{i-1} \rho_i}{\alpha_i},\\
	B_i &:= \frac{\rho_i^2}{\alpha_i^2} +\frac{\rho_i^2}{\alpha_i} .
	\end{align*}
	It is clear that $A_k, B_k \to \infty$ as $k \to \infty$. Now choose $C>0$ sufficiently large such that
	\[
	\delta(h_{\max}) A_1 + \delta(h_{\max})^2 B_1 \leq C.
	\]
	Now pick $0< h \leq h_{\max}$. Since $\delta: (0, \infty) \to \R$ is a monotonically increasing function function we get the existence of $\tilde k \in \N$, $\tilde k \geq 1$ such that
	\[
	\sum\limits_{i=1}^{\tilde k} H_i \leq C.
	\]
  Hence, the following expression is well-defined
	\[
	k(h) := \max \left\{  k \in \N: \;  \sum\limits_{i=1}^{k} H_i \leq C \right\}.
	\]
	It is left to show that $k(h) \to \infty$ as $h \to 0$. Assume that this is wrong, hence there exists a $n \in \N$ such that $k(h) < n$ for all $h > 0$.  This yields
	\[
	\sum\limits_{i=1}^{ n } H_i = \delta(h)  \sum\limits_{i=1}^{ n } A_i +  \delta(h)^2  \sum\limits_{i=1}^{ n } B_i> C \quad \forall \; 0 < h \leq h_{\max}.
	\]
	However, since $A_i$ and $B_i$ are independent from $h$ this is a contradiction for $h$ small enough. This finishes the proof.
\end{proof}

If the disretized subproblem is only solved inexactly we can establish the following result. The proof is a combination of Corollary \ref{cor:SC_1} and Corollary \ref{cor:SC_2}.

\begin{corollary}
Let $u^\dagger$ satisfy Assumption \ref{ass:SC} and let $(\eps_k)_k$ be a sequence of positive real numbers such that
\[
\sum\limits_{i=1}^\infty R_i < \infty.
\]
Let $h> 0$ be given. Then there exists a constant C such that  for every $0 < h \leq h_{\max}$ there exists a stopping index $k(h)$ such that
\[
\sum\limits_{i=1}^{k(h)} H_i \leq C < \infty.
\]
and $k(h) \to \infty$ as $h \to 0$. Furthermore $u_{k(h)}^\tin \to u^\dagger$ as $h \to 0$.
\end{corollary}

}

\subsection{Convergence under Assumption \ref{ass:ActiveSet}}
Let us now consider the case when Assumption \ref{ass:ActiveSet} is satisfied. 

{

\begin{theorem}\label{thm:convergence_ASC}
Let $u^\dagger$ satisfy Assumption \ref{ass:ActiveSet} and let $(\eps_k)_k$ be a sequence of positive real numbers. Furthermore let $h > 0$ be given and let $(u_k^\tin)_k$ be a sequence generated by Algorithm \ref{alg:MinInEx}. Then we have the estimate
\begin{align*}
\sum\limits_{i=1}^k  \frac{1}{\alpha_{i}}\|u_{i}^\tin - u^\dagger\|^2 \leq c\left( 1+ \sum\limits_{i=1}^k \frac{\gamma_{i-1}^{-\kappa}}{\alpha_i} +  \sum\limits_{i=1}^k R_i +  \sum\limits_{i=1}^k H_i \right)
\end{align*}
with the abbreviations
\begin{align*}
R_i &:= \frac{\eps_i}{\alpha_i}+ \frac{\eps_i^2}{\alpha_i^2} + \frac{\gamma_{i-1} \eps_i}{\alpha_i}  + \frac{\eps_i^2}{\alpha_i},\\
H_i &:= \delta(h)\left( \frac{\rho_i}{\alpha_i}  + \frac{\gamma_{i-1}}{\alpha_i} + \frac{\gamma_{i-1} \rho_i}{\alpha_i}  \right) + \delta(h)^2 \left(  \frac{\rho_i^2}{\alpha_i^2} + \frac{\rho_i^2}{\alpha_i}  \right).
\end{align*}
\end{theorem}

\begin{proof}
The proof mainly follows the idea of Theorem \ref{thm:convergence_SC}. Again the main part is to establish estimates for the regularization error for $u_{k+1}^\tex$. First we want to estimate the term $(u^\dagger, u^\dagger - u)$ using Assumption \ref{ass:ActiveSet}. We use \cite[Lemma 4.12]{wachsmuth2016} and obtain
\begin{equation}\label{eq:tech_1}
(u^\dagger, u^\dagger - u) \leq (S^\ast w, u^\dagger - u) + c\|u^\dagger - u\|_{L^1(A)}, \quad \forall u \in \Uad.
\end{equation}
This inequality introduces an additional $L^1$-term. To compensate this term we use an improved optimality condition, which is valid under Assumption \ref{ass:ActiveSet}
\begin{equation}\label{eq:improved_optimality}
(-p^\dagger, u-u^\dagger) \geq c_A \|u-u^\dagger\|_{L^1(A)}^{1 + \frac{1}{\kappa}}, \quad \forall u \in \Uad,
\end{equation}
with $c_A > 0$. For a proof we refer to \cite[Lemma 4.11]{wachsmuth2016}. Similar to \eqref{eq:eliminate_z} we compute
\begin{align*}
(-\lambda_k^\tin &, u^\dagger - u_{k+1}^\tex) \\
&\leq \sum\limits_{i=1}^k \frac{1}{\alpha_i} (S(u_i^\tin - u^\dagger),S(u^\dagger - u_{k+1}^\tex) )\\
&\quad + \sum\limits_{i=1}^k \frac{1}{\alpha_i} \underbrace{ (Su^\dagger - z, S(u^\dagger - u_{k+1}^\tex))} _{\leq - c_A \|u^\dagger - u_{k+1}^\tex\|_{L^1(A)}^{1 + \frac{1}{\kappa}}} + c \gamma_k \delta(h)\\
&\leq \alpha_{k+1}(-v_k^\tin, v_{k+1}^\tin - v_k^\tin) - c_A \gamma_k \|u^\dagger - u_{k+1}^\tex\|_{L^1(A)}^{1 + \frac{1}{\kappa}}\\
&\quad + c \gamma_k( \eps_{k+1} + \rho_{k+1} \delta(h) ) + c \gamma_k \delta(h).
\end{align*}
Now we estimate the term $(u^\dagger, u^\dagger - u_{k+1}^\tex)$ using Young's inequality
\begin{equation}\label{eq:asc_tech_1}
\begin{split}
\sum\limits_{i=1}^k &\frac{1}{\alpha_i} (u^\dagger, u^\dagger - u_i^\tex) \leq \sum\limits_{i=1}^k \frac{1}{\alpha_i} \big[ (S^\ast w, u^\dagger - u_i^\tex) + c \|u^\dagger - u_i^\tex\|_{L^1(A)}  \big]\\
&= \left( w, \sum\limits_{i=1}^k \frac{1}{\alpha_i} S(u^\dagger - u_i^\tin) \right) + \left(w, \sum\limits_{i=1}^k \frac{1}{\alpha_i} S(u_i^\tin - u_{i,h}^\tex)  \right)\\
&\quad + \left( w, \sum\limits_{i=1}^k \frac{1}{\alpha_i}S(u_{i,h}^\tex - u_i^\tex) \right) + c \sum\limits_{i=1}^k \frac{1}{\alpha_i} \|u^\dagger - u_i^\tex\|_{L^1(A)}\\
&\leq \|w\|_Y^2 + \frac{1}{4}\|v_k^\tin\|_Y^2 + c \sum\limits_{i=1}^k \frac{\eps_i}{\alpha_i} + c \delta(h) \sum\limits_{i=1}^k \frac{\rho_i}{\alpha_i}\\
&\quad + \frac{c_A}{2} \sum\limits_{i=1}^k \frac{\gamma_{i-1}}{\alpha_i} \|u^\dagger - u_i^\tex\|_{L^1(A)}^{1 + \frac{1}{\kappa}} + c \sum\limits_{i=1}^k \frac{\gamma_{i-1}^{-\kappa}}{\alpha_i}.
\end{split}
\end{equation}
We now consider the following inequality, similar to \eqref{eq:thminex_1}
\begin{align*}
\frac{1}{\alpha_{k+1}^2} \|S(u_{k+1} &- u^\dagger)\|_Y^2 + \frac{1}{\alpha_{k+1}}\|u_{k+1}^\tex - u^\dagger\|^2 \leq \frac{1}{\alpha_{k+1}} (u^\dagger , u^\dagger - u_{k+1}^\tex)\\
&\quad + (-v_k^\tin, v_{k+1}^\tin - v_k^\tin) - \frac{c_A \gamma_k}{\alpha_{k+1}} \|u^\dagger - u_{k+1}^\tex\|_{L^1(A)}^{1 + \frac{1}{\kappa}}\\
&\quad + c \frac{\gamma_k \eps_{k+1}}{\alpha_{k+1}} + c \delta(h) \left[  \frac{\gamma_k \rho_{k+1}}{\alpha_{k+1}} + \frac{\gamma_k}{\alpha_{k+1}}  \right],
\end{align*}
and use again the equality
\[
(-v_k^\tin, v_{k+1}^\tin - v_k^\tin) = \frac{1}{2} \|v_k^\tin\|_Y^2 - \frac{1}{2} \|v_{k+1}^\tin\|_Y^2 + \frac{1}{2}\|v_{k+1}^\tin - v_k^\tin\|_Y^2.
\]
As done in Theorem \ref{thm:convergence_SC} we obtain with $\tau > 1$ that
\begin{align*}
\frac{1}{2}\|v_{k+1}^\tin - v_k^\tin\|_Y^2 &\leq \frac{1}{2 \alpha_{k+1}^2} \left( 1+ \frac{1}{\tau} \right) \|S(u_{k+1} \tex - u^\dagger)\|_Y^2\\
&\quad+ \frac{c}{\alpha_{k+1}^2} \left( \eps_{k+1}^2 - \rho_{k+1}^2 \delta(h) \right).
\end{align*}
Combining everything now reveals with some $c_\tau > 0$
\begin{align*}
c_\tau &\sum\limits_{i=1}^k \frac{1}{\alpha_i^2} \|S(u_{k+1}^\tex - u^\dagger)\| + \sum\limits_{i=1}^k \frac{1}{\alpha_i} \|u_{k+1}^\tex - u^\dagger\|^2\\
&\quad+ c_A \sum\limits_{i=1}^k  \frac{\gamma_{i-1}}{\alpha_i} \|u^\dagger - u_{k+1}^\tex\|_{L^1(A)}^{1 + \frac{1}{\kappa}} + \frac{1}{2}\|v_k^\tin\|_Y^2 \leq \sum\limits_{i=1}^k  \frac{1}{\alpha_i} (u^\dagger, u^\dagger - u_i^\tex)\\
&\quad + c \sum\limits_{i=1}^k \left[  \frac{\eps_i^2}{\alpha_i^2} + \frac{\gamma_{i-1} \eps_i}{\alpha_i} \right] + c \delta(h) \sum\limits_{i=1}^k  \left[  \frac{\gamma_{i-1} \rho_i}{\alpha_i} + \frac{\gamma_{i-1}}{\alpha_i} \right]\\
&\quad + c \delta(h)^2 \sum\limits_{i=1}^k  \frac{\rho_i^2}{\alpha_i^2}.
\end{align*}
Now we plug in our estimate \eqref{eq:asc_tech_1} and obtain
\begin{align*}
c_\tau &\sum\limits_{i=1}^k \frac{1}{\alpha_i^2} \|S(u_{k+1}^\tex - u^\dagger)\| + \sum\limits_{i=1}^k \frac{1}{\alpha_i} \|u_{k+1}^\tex - u^\dagger\|^2\\
&\quad+ \frac{c_A}{2} \sum\limits_{i=1}^k  \frac{\gamma_{i-1}}{\alpha_i} \|u^\dagger - u_{k+1}^\tex\|_{L^1(A)}^{1 + \frac{1}{\kappa}} + \frac{1}{4}\|v_k^\tin\|_Y^2 \leq \|w\|_Y^2 + c \sum\limits_{i=1}^k \frac{\gamma_{i-1}^{-\kappa}}{\alpha_i}\\
&\quad + c \sum\limits_{i=1}^k \left[  \frac{\eps_i^2}{\alpha_i^2} + \frac{\gamma_{i-1} \eps_i}{\alpha_i} + \frac{\eps_i}{\alpha_i}\right] + c \delta(h) \sum\limits_{i=1}^k  \left[ \frac{\rho_i}{\alpha_i} + \frac{\gamma_{i-1} \rho_i}{\alpha_i} + \frac{\gamma_{i-1}}{\alpha_i} \right]\\
&\quad + c \delta(h)^2 \sum\limits_{i=1}^k  \frac{\rho_i^2}{\alpha_i^2}.
\end{align*}
As in the proof of Theorem \ref{thm:convergence_SC} we apply triangular inequality to finish the proof.

\end{proof}

Let us now establish convergence results similar to Corollary \ref{cor:SC_1} and \ref{cor:SC_2}.

\begin{corollary}
	Let $u^\dagger$ satisfy Assumption \ref{ass:ActiveSet} and let $(\eps_k)_k$ be a sequence of positive real numbers such that $\gamma_{i-1} \eps_i \to 0$.	Furthermore assume that $S_h = S$ and let $(u_k^\tin)_k$ be a sequence generated by Algorithm \ref{alg:MinInEx}. Then we obtain 
	\[
	\min\limits_{i=1,..,k} \|u_i^\tin - u^\dagger\| \to 0
	\]
	as $k \to \infty$.
\end{corollary}

\begin{proof}
The sequence $(\alpha_k)_k$ is bounded by a constant $M$. Hence we have the following inequalities for $k$ large enough
\[
\sum\limits_{i=2}^k \frac{\gamma_{i-1}}{\alpha_i} \eps_i \geq \frac{1}{M} \sum\limits_{i=2}^k \frac{i-1}{\alpha_i} \eps_i \geq \frac{1}{M} \sum\limits_{i=2}^k \frac{\eps_i}{\alpha_i}.
\]
Furthermore we have by \cite[Lemma 3.5]{wachsmuth2016} that
\[
\gamma_k^{-1}  \sum\limits_{i=1}^k \frac{\gamma_{i-1}^{-\kappa}}{\alpha_i} +  \gamma_k^{-1} \sum\limits_{i=1}^k \frac{\gamma_{i-1}}{\alpha_i} \eps_i \to 0.
\]
We now obtain
\[
\min\limits_{i=1,..,k} \|u_i^\tin - u^\dagger\|  \leq c \left(  \gamma_k^{-1} + \gamma_k^{-1} \sum\limits_{i=1}^k \frac{\gamma_{i-1}^{-\kappa}}{\alpha_i} + \gamma_k^{-1} \sum\limits_{i=1}^k R_i  \right) \to 0,
\]
which finishes the proof.
\end{proof}

\begin{corollary}
	Let $u^\dagger$ satisfy Assumption \ref{ass:ActiveSet}. Let $h_{\max} > 0$ be given and $\eps_k = 0$ for all $k \in \N$. Then there exists a constant C such that  for every $0 < h \leq h_{\max}$ there exists a stopping index $k(h)$ such that
	\[
	\sum\limits_{i=1}^{k(h)} H_i \leq C < \infty.
	\]
	and $k(h) \to \infty$ as $h \to 0$. Furthermore 
	\[
	\min\limits_{i=1,..,k(h)} \|u_i^\tin - u^\dagger\| \to 0
	\]
	as $h \to 0$.
\end{corollary}
\begin{proof}
The proof is very similar to the proof of Corollary \ref{cor:SC_2}.
\end{proof}

A combination of both results yields the following corollary.
\begin{corollary}\label{cor:ASC_1}
	Let $u^\dagger$ satisfy Assumption \ref{ass:ActiveSet} and let $(\eps_k)_k$ be a sequence of positive real numbers such that $\gamma_{i-1} \eps_i \to 0$.	Let $h> 0$ be given.  Then there exists a constant C such that  for every $0 < h \leq h_{\max}$ there exists a stopping index $k(h)$ such that
	\[
	\sum\limits_{i=1}^{k(h)} H_i \leq C < \infty.
	\]
	and $k(h) \to \infty$ as $h \to 0$. Furthermore 
	\[
	\min\limits_{i=1,..,k(h)} \|u_i^\tin - u^\dagger\| \to 0
	\]
	as $h \to 0$.
\end{corollary}

}

\section{{Numerical example}}\label{sec:numerics}

Now, let $Sy=u$ be defined as the (weak) solution of the linear partial differential equation {for a convex set $\Omega \subset \R^n$ ($n=2,3$)}
\begin{equation}\label{eq:PDE}
\begin{alignedat}{3}
 -\Delta y &= u \;\; &&\text{in } \Omega,\\
 y&=0 &&\text{on } \partial \Omega.
\end{alignedat}
\end{equation}
{Let us show that this example fit into our framework. Clearly, for $u \in L^2(\Omega)$ equation \eqref{eq:PDE} has a unique weak solution $y \in H_0^1(\Omega)$, and the associated solution operator $S$ is linear and continuous. For the choice $Y = L^2(\Omega) $ we obtain $S^\ast = S$.\\
	
Let us now report on the discretization and the operator $S_h$. We follow the argumentation and results presented in \cite[Section 3]{wachsmuth2013adaptive}. Let $\mathcal{T}_h$ be a regular mesh which consists of closed cells $T$. For $T \in \mathcal{T}_h$ we define $h_T := \text{diam } T$. Furthermore we set $h := \max_{T \in \mathcal{T}_h} h_T$. We assume that there exists a constant $R > 0$ such that $\frac{h_T}{R_T} \leq R$ for all $T \in \mathcal{T}$. Here  we define $R_T$ to be the diameter of the largest ball contained in $T$.\\
For this mesh $\mathcal{T}$ we define an associated finite dimensional space $Y_h \subset H_0^1(\Omega)$, such that the restriction of a function $v \in Y_h$ to a cell $T \in \mathcal{T}$ is a linear polynomial.\\
The operator $S_h$ is now defined in the sense of weak solutions. We set $y_h := S_h u$ if $y_h \in Y_h$ solves
\[
\int\limits_{\Omega} \nabla y_h \cdot \nabla v_h \; dx = \int\limits_{\Omega} u \cdot v_h \; dx \quad \forall v_h \in Y_h.
\]
We also obtain $S_h^\ast = S_h$ in the discrete case. Let us now mention that the operator $S_h$ satisfy Assumption \ref{ass:operator_Sh}. Following \cite{wachsmuth2013adaptive} and the references therein we obtain the following result.

\begin{lemma}
Assume that there exists a constant $C_M > 1$ such that $\max\limits_{T \in \mathcal{T}_h} h_T \leq C_M \min\limits_{T \in \mathcal{T}_h}  h_T$ holds. Then we have the estimates
\begin{align*}
\|(S-S_h) f\|_{L^2(\Omega)} &\leq c h^2 \|f\|_{L^2(\Omega)},\\
\|(S^\ast - S_h^\ast) f\|_{L^\infty(\Omega)} & \leq c h^{2 - n/2} \|f\|_{L^2(\Omega)},
\end{align*}
for $f \in L^2(\Omega)$ and a constant $c$ independent from $f$ and $h$.
\end{lemma}
Hence Assumption \ref{ass:operator_Sh} is satisfied with $\delta(h) = ch^2$. \\

Let us quickly resort on the computation of the solution of \eqref{eq:subprob_abstract}. In \cite[Section 4]{poerner2016b} we applied a variational discretization and a semi-smooth Newton solver to this problem. The space $Y_h$ was defined as the span of linear finite elements. This gives us approximate solutions $(u_h, y_h, p_h)$ such that $y_h := S_h u_h \in Y_h$, $p_h := S_h^\ast (z - y_h) \in Y_h$ and $u_h \in \Uad$. Here the control $u_h$ can be computed as the truncation of a finite element. For more details we refer to \cite{poerner2016b,hinze2012,beuchler2012} and the references therein.

We now consider the following optimal control problem. Note that due to the linearity of $S$ this is of  form \eqref{eq:main_problem}.}
\begin{equation}\label{eq:test_problem}
 \begin{split}
    \text{Minimize} &\quad \frac{1}{2}\|y - z\|^2  \\
    \text{such that} & \quad -\Delta y = u + e_\Omega \quad \text{in } \Omega,\\
    &\quad y=0 \quad \text{on } \partial \Omega,\\
    &\quad u_a \leq u \leq u_b \quad \text{a.e. in } \Omega.
\end{split}
\end{equation}

We use the {inexact Bregman method \ref{alg:MinInEx}} to solve \eqref{eq:test_problem}. With the choice of $\Omega = (0,1)^2$, $u_a = -1$, $u_b = 1$ and
\begin{align*}
p^\dagger(x) &= - \frac{1}{8 \pi^2} \sin(2\pi x) \sin(2\pi y)\\
u^\dagger(x) &= - \sign(p^\dagger(x))\\
y^\dagger(x) &= \sin(\pi x) \sin(\pi y)\\
e_\Omega(x) &= 2 \pi^2 \sin(\pi x) \sin(\pi y) - u^\dagger\\
z(x) &= \sin(\pi x) \sin(\pi y) + \sin(2\pi x) \sin(2\pi y)
\end{align*}
the functions $(u^\dagger, y^\dagger, p^\dagger)$ are a solution to \eqref{eq:test_problem}. Here the solution satisfies assumption \ref{ass:ActiveSet} with $A=\Omega$ and $\kappa = 1$. {We use different  mesh sizes for comparison and plot the error for the first $500$ iterations in Figure \ref{fig:error1}, \ref{fig:error2} and \ref{fig:error3}. Furthermore we set $\alpha_k := 0.1$ and $\eps_k :=  k^{-3/2}$ to satisfy the assumptions of Corollary \ref{cor:ASC_1}. As expected we see that for $h \to 0$ we obtain convergence for $k \to \infty$. The coarsest mesh has $10^2$ and the finest mesh has approximately $10^5$ degrees of freedom.}

\begin{figure}[htbp]
\makebox[\textwidth]{\input{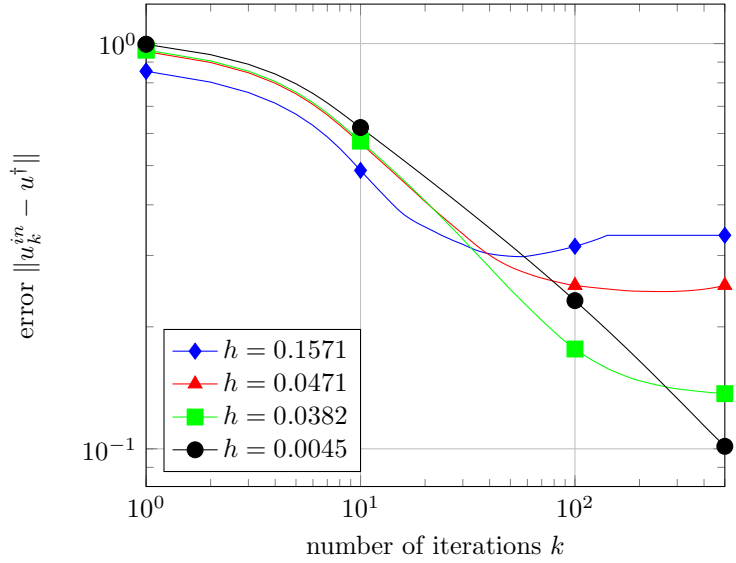} }
\caption{Error $\|u_k^\tin - u^\dagger\|$ for different mesh sizes $h$.}
\label{fig:error1}
\end{figure}

\begin{figure}[htbp]
	\makebox[\textwidth]{\input{bb_y.tikz} }
	\caption{Error $\|y_k^\tin - y^\dagger\|$ for different mesh sizes $h$.}
	\label{fig:error2}
\end{figure}

\begin{figure}[htbp]
	\makebox[\textwidth]{\input{bb_p.tikz} }
	\caption{Error $\|p_k^\tin - p^\dagger\|$ for different mesh sizes $h$.}
	\label{fig:error3}
\end{figure}

\section{Conclusion}\label{sec:conclusion}
We showed that our iterative method is robust against numerical errors. {
Furthermore we established error estimates and convergence result both for errors introduced by the accuracy of the computed iterates and by the discretization. We constructed an a-posteriori error estimator for the discretized subproblem and provided numerical results.\\}
Together with the exact a-priori regularization estimates \cite{wachsmuth2016} and the convergence results obtained for noisy data \cite{poerner2016b}, we conclude that the Bregman iterative method is a stable and robust method to compute solutions for our model problem \eqref{eq:main_problem}.

\bibliographystyle{plain}
\bibliography{literatur}

\end{document}